\documentclass{article}
\usepackage{amsfonts}
\usepackage{amsmath}

\newtheorem{theorem}{Theorem}

\newtheorem{corollary}[theorem]{Corollary}

\newenvironment{proof}[1][Proof]{\noindent\textbf{#1.} }{\ \rule{0.5em}{0.5em}}
\input{tcilatex}
\begin{document}

\title{\textbf{On properties of Tribonacci-Lucas polynomials}}
\author{\textbf{Hasan Kose, Nazmiye Yilmaz, Necati Taskara} \\
Department of Mathematics, Science Faculty,\\
Selcuk University, Campus, 42250, Konya, Turkey\medskip \\
\textit{hkose@selcuk.edu.tr, nzyilmaz@selcuk.edu.tr, ntaskara@selcuk.edu.tr}}
\maketitle

\begin{abstract}
In this paper, we investigated properties of Tribonacci-Lucas polynomials
which generalized Tribonacci-Lucas numbers. From this generalization, we
also obtain some new algebraic properties on these numbers and polynomials
as Binet formula, summation, binomial sum and generating function.


\end{abstract}

\section{Introduction}

 Recently, Fibonacci and Lucas numbers have investigated very largely
and authors tried to developed and give some directions to mathematical
calculations using these type of special numbers. One of these directions
goes through to the \textit{Tribonacci }and the \textit{Tribonacci-Lucas
numbers. }In fact Tribonacci numbers have been firstly defined by M.
Feinberg in 1963 and then some important properties over this numbers have
been created by [3,6-10,13]. On the other hand, Tribonacci-Lucas numbers
have been introduced and investigated by author in [2,4]. In addition, there
exists another mathematical term, namely to be incomplete, on Tribonacci and
Tribonacci-Lucas numbers. As a brief background, the incomplete Tribonacci
and Tribonacci-Lucas numbers were introduced by authors [11,12], and further
the generating functions of these numbers were presented by authors.

For $n\geq 2$, it is known that while the Tribonacci sequence $\left\{
T_{n}\right\} _{n\in 
\mathbb{N}
}$ is defined by 
\begin{equation}
T_{n+1}=T_{n}+T_{n-1}+T_{n-2}\ \ \ (T_{0}=0,\ T_{1}=T_{2}=1),  \label{1.1}
\end{equation}%
and the Tribonacci-Lucas sequence $\left\{ K_{n}\right\} _{n\in 
\mathbb{N}
}$ is defined by

\begin{equation}
K_{n+1}=K_{n}+K_{n-1}+K_{n-2}\ \ (K_{0}=3,\ K_{1}=1,\ K_{2}=3).  \label{1.2}
\end{equation}%
There is also well known that each of the Tribonacci and Tribonaccci-Lucas
numbers is actually a linear combination of $\alpha ^{n}$, $\beta ^{n}$ and $%
\gamma ^{n}.$ In other words, 
\begin{equation}
\left. 
\begin{array}{c}
T_{n}=\frac{\alpha ^{n+1}}{(\alpha -\beta )(\alpha -\gamma )}+\frac{\beta
^{n+1}}{(\beta -\alpha )(\beta -\gamma )}+\frac{\gamma ^{n+1}}{(\gamma
-\alpha )(\gamma -\beta )} \\ 
\text{and} \\ 
K_{n}=\alpha ^{n}+\beta ^{n}+\gamma ^{n},%
\end{array}%
\right\}  \label{1.3}
\end{equation}%
where $\alpha ,\beta $ and $\gamma $ are roots of the characteristic
equations of (\ref{1.1}) and (\ref{1.2})\ such that 
\begin{eqnarray*}
\alpha &=&\frac{1+\sqrt[3]{19+3\sqrt{33}}+\sqrt[3]{19-3\sqrt{33}}}{3},\
\beta =\frac{1+w\sqrt[3]{19+3\sqrt{33}}+w^{2}\sqrt[3]{19-3\sqrt{33}}}{3}, \\
\ \gamma &=&\frac{1+w^{2}\sqrt[3]{19+3\sqrt{33}}+w\sqrt[3]{19-3\sqrt{33}}}{3}%
,
\end{eqnarray*}%
where $w=\frac{-1+i\sqrt{3}}{2}.$

Meanwhile we note that equations in (\ref{1.3}) are called the Binet
formulas for Tribonacci and Tribonacci-Lucas numbers, respectively.

Moreover, authors studied a large class of polynomials by Fibonacci and
Tribonacci numbers [1,5]. The Fibonacci polynomials are defined\ by%
\begin{equation*}
F_{n+2}\left( x\right) =xF_{n+1}\left( x\right) +F_{n}\left( x\right) ,
\end{equation*}%
where initial conditions $F_{0}\left( x\right) =0$ and $F_{1}\left( x\right)
=1$. And, in 1973, Hoggatt and Bicknell \cite{5} introduced Tribonacci
polynomials. The Tribonacci polynomials $T_{n}\left( x\right) $ are defined
by the recurrence relation%
\begin{equation}
T_{n+3}\left( x\right) =x^{2}T_{n+2}\left( x\right) +xT_{n+1}\left( x\right)
+T_{n}\left( x\right) ,  \label{1.4}
\end{equation}%
where $T_{0}\left( x\right) =0,\ T_{1}\left( x\right) =1,\ T_{2}\left(
x\right) =x^{2}.$ Note that $T_{n}\left( 1\right) =T_{n},$ $n\in 
\mathbb{N}
$. In addition to, they gave the binomial sums of these polynomials as%
\begin{equation*}
F_{n}\left( x\right) =\sum_{j=0}^{\left[ \frac{n-1}{2}\right] }\binom{n-j-1}{%
j}x^{n-2j-1},
\end{equation*}%
where $\left[ x\right] $ is the greatest integer contained in $x$, and 
\begin{equation*}
T_{n}\left( x\right) =\sum_{j=0}\binom{n-j-1}{j}_{3}x^{2n-3j-2},
\end{equation*}%
where $\binom{n}{j}_{3}$ is the trinomial coefficient in the $n^{th}$ row
and $j^{th}$ column where, as is usual, the left-most column is the zero$%
^{th}$ column and the top row is the zero$^{th}$ row, and $\binom{n}{j}_{3}=0
$ if $j>n.$

Also, in [11], authors defined Tribonacci-Lucas polynomials, incomplete
Tribonacci-Lucas numbers and incomplete Tribonacci-Lucas polynomials. That
is, Tribonacci-Lucas polynomials are defined by 
\begin{equation}
K_{n+3}\left( x\right) =x^{2}K_{n+2}\left( x\right) +xK_{n+1}\left( x\right)
+K_{n}\left( x\right) ,  \label{1.5}
\end{equation}%
where $K_{0}\left( x\right) =3,\ K_{1}\left( x\right) =x^{2},\ K_{2}\left(
x\right) =x^{4}+2x$.

Incomplete Tribonacci-Lucas numbers and polynomials are defined by%
\begin{equation*}
K_{n}^{\left( s\right) }\left( x\right)
==\sum\limits_{i=0}^{s}\sum\limits_{j=0}^{i}\frac{n}{n-i-j}\dbinom{i}{j}%
\dbinom{n-i-j}{i}x^{2n-3\left( i+j\right) },
\end{equation*}%
and%
\begin{equation*}
K_{n}\left( s\right) ==\sum\limits_{i=0}^{s}\sum\limits_{j=0}^{i}\frac{n}{%
n-i-j}\dbinom{i}{j}\dbinom{n-i-j}{i},
\end{equation*}%
respectively, where $0\leq s\leq \left\lfloor \frac{n}{2}\right\rfloor $ for 
$n\in 
\mathbb{Z}
^{+}$.

In the light of the above paragraph, the main goal of this paper is to
improve the Tribonacci-Lucas polynomials and numbers\ with a different
viewpoint. In order to do that we first obtain Binet formula, sum and
binomial summation of Tribonacci-Lucas polynomials. Then\textit{, }we give%
\textit{\ }the generating function of this polynomial.\qquad 

\section{Main Results}

In this section, we will mainly focus on the Tribonacci-Lucas polynomials to
get some important results. In fact, we will present the related Binet
formula, relationship of between Tribonacci polynomials, summation, binomial
sum and generating function. Besides, we will get similar result for the
Tribonacci-Lucas numbers by using this polynomials.

Hence, we firstly derive the Binet formula of the Tribonacci-Lucas
polynomials.

\begin{theorem}
For $n\in 
\mathbb{N}
,$ we can write the Binet formulas for the Tribonacci-Lucas polynomials as
the form 
\begin{equation*}
K_{n}\left( x\right) =A\lambda _{1}^{n}+B\lambda _{2}^{n}+C\lambda _{3}^{n},
\end{equation*}%
where 
\begin{eqnarray*}
A &=&\dfrac{K_{2}\left( x\right) \mathcal{-}\left( \lambda _{2}+\lambda
_{3}\right) K_{1}\left( x\right) +\lambda _{2}\lambda _{3}K_{0}\left(
x\right) }{\left( \lambda _{1}-\lambda _{2}\right) \left( \lambda
_{1}-\lambda _{3}\right) },~ \\
B &=&\dfrac{K_{2}\left( x\right) \mathcal{-}\left( \lambda _{1}+\lambda
_{3}\right) K_{1}\left( x\right) +\lambda _{1}\lambda _{3}K_{0}\left(
x\right) }{\left( \lambda _{2}-\lambda _{1}\right) \left( \lambda
_{2}-\lambda _{3}\right) },~ \\
C &=&\dfrac{K_{2}\left( x\right) \mathcal{-}\left( \lambda _{1}+\lambda
_{2}\right) K_{1}\left( x\right) +\lambda _{1}\lambda _{2}K_{0}\left(
x\right) }{\left( \lambda _{3}-\lambda _{1}\right) \left( \lambda
_{3}-\lambda _{2}\right) },
\end{eqnarray*}%
such that $\lambda _{1},~\lambda _{2},~\lambda _{3}$ are roots of
characteristic equation of (\ref{1.5}).
\end{theorem}

\begin{proof}
We will show the thruthness of the Binet formula for Tribonacci-Lucas
polynomials.

So let us consider (\ref{1.5}). By the assumption, the roots of the
characteristic equation of (\ref{1.5}) are $\lambda _{1},~\lambda _{2}$ and $%
\lambda _{3}$. Hence the general solution of it is given by 
\begin{equation*}
K_{n}\left( x\right) =C_{1}\lambda _{1}^{n}+C_{2}\lambda
_{2}^{n}+C_{3}\lambda _{3}^{n}.
\end{equation*}%
Using initial conditions of equation (\ref{1.5}) and also applying
fundamental linear algebra operations, we clearly get the coefficient $%
C_{1}=A$, $C_{2}=B$ and $C_{3}=C$, as desired. This implies the formula for $%
K_{n}\left( x\right) $.
\end{proof}

\begin{theorem}
The relation between of \ theTribonacci polynomials $T_{n}\left( x\right) $
and the Tribonacci-Lucas polynomials $K_{n}\left( x\right) $ is%
\begin{equation*}
K_{n}\left( x\right) =x^{2}T_{n}\left( x\right) +2xT_{n-1}\left( x\right)
+3T_{n-2}\left( x\right) ,
\end{equation*}%
where $n\geq 2.$
\end{theorem}

\begin{proof}
Let us show this by induction, for $n=2,$ we can write 
\begin{eqnarray*}
K_{2}\left( x\right) &=&x^{2}T_{2}\left( x\right) +2xT_{1}\left( x\right)
+3T_{0}\left( x\right) \\
&=&x^{4}+2x.
\end{eqnarray*}
Now, assume that, it is true for all positive integers $m$, i.e. 
\begin{equation}
K_{m}\left( x\right) =x^{2}T_{m}\left( x\right) +2xT_{m-1}\left( x\right)
+3T_{m-2}\left( x\right) .  \label{2.0}
\end{equation}%
Then, we need to show that above equality holds for $n=m+1,$ that is, 
\begin{equation}
K_{m+1}\left( x\right) =x^{2}T_{m+1}\left( x\right) +2xT_{m}\left( x\right)
+3T_{m-1}\left( x\right) .  \label{2.1}
\end{equation}%
By considering the left hand side of Equation (\ref{2.1}), we can expand the
recurrence relation as 
\begin{equation*}
K_{m+1}\left( x\right) =x^{2}K_{m}\left( x\right) +xK_{m-1}\left( x\right)
+K_{m-2}\left( x\right) .
\end{equation*}%
Then, using Equation (\ref{2.0}), we have%
\begin{eqnarray*}
K_{m+1}\left( x\right) &=&x^{2}\left( x^{2}T_{m}\left( x\right)
+2xT_{m-1}\left( x\right) +3T_{m-2}\left( x\right) \right) \\
&&+x\left( x^{2}T_{m-1}\left( x\right) +2xT_{m-2}\left( x\right)
+3T_{m-3}\left( x\right) \right) \\
&&+\left( x^{2}T_{m-2}\left( x\right) +2xT_{m-3}\left( x\right)
+3T_{m-4}\left( x\right) \right)
\end{eqnarray*}%
Finally, by considering (\ref{1.4}), we obtain%
\begin{equation*}
K_{m+1}\left( x\right) =x^{2}T_{m+1}\left( x\right) +2xT_{m}\left( x\right)
+3T_{m-1}\left( x\right)
\end{equation*}%
which ends up the induction.
\end{proof}

In \cite{4}, the author obtained the relationship for Tribonacci and
Tribonacci-Lucas numbers. However, in here, we will obtain this relationship
in terms of Tribonacci-Lucas polynomials as a consequence of Theorem 2. To
do that we will take $x=1$ in Theorem 2.

\begin{corollary}
The relation between of Tribonacci numbers $T_{n}$ and Tribonacci-Lucas
numbers $K_{n}$ is%
\begin{equation*}
K_{n}=T_{n}+2T_{n-1}+3T_{n-2},
\end{equation*}%
where $n\geq 2.$
\end{corollary}

Now, we will give binomial summation of Tribonacci-Lucas polynomials as
follows:

\begin{theorem}
For $n\in 
\mathbb{N}
,$ we have the equality%
\begin{equation}
K_{3n}\left( x\right) =\sum\limits_{i=0}^{n}\sum\limits_{j=0}^{i}\dbinom{n}{i%
}\dbinom{i}{j}x^{i+j}K_{i+j}\left( x\right) .  \label{2.2}
\end{equation}
\end{theorem}

\begin{proof}
Let $\lambda $ stand for a root of the characteristic equation of equation (%
\ref{1.5}). Then, we have%
\begin{equation*}
\lambda ^{3}=x^{2}\lambda ^{2}+x\lambda +1
\end{equation*}%
and we can write by considering binomial expansion%
\begin{eqnarray*}
\left( \lambda ^{3}\right) ^{n} &=&\left( \lambda ^{3}-1+1\right) ^{n} \\
&=&\sum\limits_{i=0}^{n}\dbinom{n}{i}\left( \lambda ^{3}-1\right) ^{i} \\
&=&\sum\limits_{i=0}^{n}\dbinom{n}{i}\left( x^{2}\lambda ^{2}+x\lambda
\right) ^{i} \\
&=&\sum\limits_{i=0}^{n}\dbinom{n}{i}\left( x\lambda \right)
^{i}\sum\limits_{j=0}^{i}\dbinom{i}{j}\left( x\lambda \right) ^{j} \\
&=&\sum\limits_{i=0}^{n}\sum\limits_{j=0}^{i}\dbinom{n}{i}\dbinom{i}{j}%
\left( x\lambda \right) ^{i+j}
\end{eqnarray*}%
If we replace to $\lambda _{1},~\lambda _{2},~\lambda _{3}$ by $\lambda $
and rearrange, then we obtain%
\begin{equation*}
K_{3n}\left( x\right) =\sum\limits_{i=0}^{n}\sum\limits_{j=0}^{i}\dbinom{n}{i%
}\dbinom{i}{j}x^{i+j}K_{i+j}\left( x\right) .
\end{equation*}
\end{proof}

In \cite{13}, the authors presented the binomial sum for Tribonacci-Lucas
numbers. However, in here, we will derive this sum by using Tribonacci-Lucas
polynomials. To do that we will take $x=1$ in Theorem 4. Hence we have the
following corollary.

\begin{corollary}
For $n\in 
\mathbb{N}
,$ we have the following equality;%
\begin{equation}
K_{3n}=\sum\limits_{i=0}^{n}\sum\limits_{j=0}^{i}\dbinom{n}{i}\dbinom{i}{j}%
K_{i+j}.  \label{2.3}
\end{equation}%
\qquad \qquad \qquad
\end{corollary}

For Tribonacci-Lucas polynomials, we give the summations according to
specified rules as we depicted at the beginning of this section.

\begin{theorem}
For $n>0$ and $\ m>j\geq 0,$there exist%
\begin{eqnarray}
\sum_{i=0}^{n-1}K_{mi+j}\left( x\right) &=&\frac{K_{mn+j+m}\left( x\right)
+K_{mn+j-m}\left( x\right) +\left( 1-X_{-m}\right) K_{mn+j}\left( x\right) }{%
X_{-m}-X_{m}}  \label{2.4} \\
&&-\frac{K_{j+m}\left( x\right) +K_{m-j}\left( x\right) +\left(
1-X_{-m}\right) K_{j}\left( x\right) }{X_{-m}-X_{m}}  \notag
\end{eqnarray}%
where $X_{-m}=\lambda _{1}^{m}+\lambda _{2}^{m}+\lambda _{3}^{m}.$
\end{theorem}

\begin{proof}
The main point of the proof will be touched just the result Theorem 1, in
other words the Binet formulas of related polynomials. Thus 
\begin{eqnarray*}
\sum\limits_{i=0}^{n-1}K_{mi+j}\left( x\right)
&=&\sum\limits_{i=0}^{n-1}\left( A\lambda _{1}^{mi+j}+B\lambda
_{2}^{mi+j}+C\lambda _{3}^{mi+j}\right) , \\
&=&A\lambda _{1}^{j}\left( \frac{\lambda _{1}^{mn}-1}{\lambda _{1}^{m}-1}%
\right) +B\lambda _{2}^{j}\left( \frac{\lambda _{2}^{mn}-1}{\lambda
_{2}^{m}-1}\right) +C\lambda _{3}^{j}\left( \frac{\lambda _{3}^{mn}-1}{%
\lambda _{3}^{m}-1}\right) \,.
\end{eqnarray*}%
In here, simplifying the last equality in above will be implied (\ref{2.4})
as required.
\end{proof}

If we choose specific values for $m$ and $j,$ then the following result will
be clear for the summation of Tribonacci-Lucas polynomials as a consequence
of Theorem 6.

\begin{corollary}
For $n>0,$ $m=1$ and $j=0$, the general sum of Tribonacci-Lucas polynomials%
\begin{equation*}
\sum_{i=0}^{n-1}K_{i}\left( x\right) =\frac{K_{n+1}\left( x\right)
+K_{n-1}\left( x\right) +\left( 1-x^{2}\right) K_{n}\left( x\right)
+2x^{2}+x-3}{x^{2}+x}
\end{equation*}%
and for $m=2$ and $j=0$, the sum with even indices of Tribonacci-Lucas
polynomials is%
\begin{equation*}
\sum_{i=0}^{n-1}K_{2i}\left( x\right) =\frac{K_{2n+2}\left( x\right)
+K_{2n-2}\left( x\right) +\left( 1-X_{-2}\right) K_{2n}\left( x\right)
+3X_{-2}-2x^{4}-4x-3}{X_{-2}-X_{2}},
\end{equation*}%
and for $m=2$ and $j=1$, the sum with odd indices of Tribonacci-Lucas
polynomials is%
\begin{equation*}
\sum_{i=0}^{n-1}K_{2i+1}\left( x\right) =\frac{K_{2n+3}\left( x\right)
+K_{2n-1}\left( x\right) +\left( 1-X_{-2}\right) K_{2n+1}\left( x\right)
+x^{2}X_{-2}-x^{6}-3x^{3}-2x^{2}-3}{X_{-2}-X_{2}},
\end{equation*}%
where $X_{-2}=\lambda _{1}^{2}+\lambda _{2}^{2}+\lambda _{3}^{2}.$
\end{corollary}

As we noted in the beginning of this section, the other aim of this paper is
to present generating function of this polynomial.

\begin{theorem}
For Tribonacci-Lucas polynomials, we have the generating function%
\begin{equation}
G\left( z\right) =\frac{3-2x^{2}z-xz^{2}}{1-x^{2}z-xz^{2}-z^{3}}.
\label{2.5}
\end{equation}
\end{theorem}

\begin{proof}
Assume that $G(z)$ is the generating function for the polynomials $\left\{
K_{n}\left( x\right) \right\} _{n\in \mathbb{N}}$. Then we have%
\begin{equation}
G\left( z\right) =K_{0}\left( x\right) +K_{1}\left( x\right) z+K_{2}\left(
x\right) z^{2}+\ldots +K_{n}\left( x\right) z^{n}+\ldots .  \label{2.6}
\end{equation}%
If we multiply the $G(z)$ given in (\ref{2.6}) with $x^{2}z,$ $xz^{2}$ and $%
z^{3},$ respectively, then we get%
\begin{equation}
\left. 
\begin{array}{c}
x^{2}zG\left( z\right) =x^{2}K_{0}\left( x\right) z+x^{2}K_{1}\left(
x\right) z^{2}+x^{2}K_{2}\left( x\right) z^{3}+\ldots +x^{2}K_{n}\left(
x\right) z^{n+1}+\ldots  \\ 
xz^{2}G\left( z\right) =xK_{0}\left( x\right) z^{2}+xK_{1}\left( x\right)
z^{3}+xK_{2}\left( x\right) z^{4}+\ldots +xK_{n}\left( x\right)
z^{n+2}+\ldots  \\ 
z^{3}G\left( z\right) =K_{0}\left( x\right) z^{3}+K_{1}\left( x\right)
z^{4}+K_{2}\left( x\right) z^{5}+\ldots +K_{n}\left( x\right) z^{n+3}+\ldots 
\end{array}%
\right\} .  \label{2.7}
\end{equation}%
Consequently, by subtracting the sum of (\ref{2.7}) from (\ref{2.6}), it is
obtained the equation%
\begin{equation*}
G\left( z\right) =\frac{K_{0}\left( x\right) +z\left( K_{1}\left( x\right)
-x^{2}K_{0}\left( x\right) \right) +z^{2}\left( K_{2}\left( x\right)
-x^{2}K_{1}\left( x\right) -xK_{0}\left( x\right) \right) }{%
1-x^{2}z-xz^{2}-z^{3}}
\end{equation*}%
which completes the proof of the Theorem.\medskip 
\end{proof}

In \cite{4}, the author obtained the generating function for the
Tribonacci-Lucas numbers. However, in here, we will obtain this function in
terms of the Tribonacci-Lucas polynomials as a consequence of the Theorem 8.
To do that we will again take $x=1$ Theorem 8. Hence we have the following
corollary.

\begin{corollary}
The generating function of the Tribonacci-Lucas numbers $K_{n}$ is given by%
\begin{equation*}
g\left( z\right) =\sum_{n=0}^{\infty }K_{n}z^{n}=\frac{3-2z-z^{2}}{%
1-z-z^{2}-z^{3}}.
\end{equation*}
\end{corollary}

\end{document}